\documentclass{amsart}
\usepackage{graphicx}
\newtheorem{thm}{Theorem}[section]

\newtheorem{lem}[thm]{Lemma}

\theoremstyle{definition}
\newtheorem{defn}[thm]{Definition}
\theoremstyle{remark} \theoremstyle{Proof}

\numberwithin{equation}{section}

\author[O. Ajebbar]{ Ajebbar Omar}
\address{ Ajebbar Omar\\Department of Mathematics\\
Ibn Zohr University, Faculty of Sciences, Agadir\\
Morocco} \email{omar-ajb@hotmail.com}
\author[E. Elqorachi]{ Elqorachi  Elhoucien}
\address{ Elqorachi Elhoucien\\Department of Mathematics\\
Ibn Zohr University, Faculty of Sciences, Agadir\\
Morocco} \email{elqorachi@hotmail.com}

\begin{document}

\title[Variants of Wilson's functional equation]{Variants of Wilson's functional equation on semigroups}


\begin{abstract}
Given a semigroup $S$ generated by its squares equipped with an
involutive automorphism $\sigma$ and a multiplicative function
$\mu:S\to\mathbb{C}$ such that $\mu(x\sigma(x))=1$ for all $x\in S$,
we determine the complex-valued solutions of the following
functional equations
\begin{equation*}f(xy)+\mu(y)f(\sigma(y)x)=2f(x)g(y),\, x,y\in
S\end{equation*} and
\begin{equation*}f(xy)+\mu(y)f(\sigma(y)x)=2f(y)g(x),\, x,y\in
S\end{equation*}
\end{abstract}
\maketitle
\section{Introduction}
In \cite{Wilson1} Wilson dealt with functional equations related to
and generalizing the cosine functional equation
$g(x+y)+g(x-y)=2g(x)g(y),\,x,y\in\mathbb{R}$. He generalized the
cosine functional equation to
\begin{equation}\label{EQ}f(x+y)+f(x-y)=2f(x)g(y),\,x,y\in \mathbb{R},\end{equation}
that contains two unknown functions $f$ and
$g$. In \cite{Wilson2} he studied his second generalization
$$f(x+y)+f(x-y)=2h(x)k(y),\,x,y\in \mathbb{R},$$ that contains the
three unknown functions $f$, $h$ and $k$.\\
Acz\'{e}l's monograph \cite{Acz�l} contains references to earlier
works on (\ref{EQ}).\\
In \cite{Ebanks and Stetkaer1} Ebanks and Stetk{\ae}r studied the
solution of Wilson's functional equation
$$f(xy)+f(xy^{-1})=2f(x)g(y),\,x,y\in G$$ and of the following variant
of Wilson's functional equation
$$f(xy)+f(y^{-1}x)=2f(x)g(y),\,x,y\in G$$ on groups,
where $f,g:G\to\mathbb{C}$ are the unknown functions. We refer also to the recent study by Stetk\ae r \cite{ST}.\\
Elqorachi and Redouani \cite{Elqorachi and Redouani} solved the
following variant of Wilson's functional equation
\begin{equation*}f(xy)+\chi(y)f(\sigma(y)x)=2f(x)g(y),x,y\in G,\end{equation*}
where $G$ is a group, $\sigma$ is an involutive automorphism of $G$
and $\chi$ is a character of $G$.
\par In the present paper we solve the following variants of Wilson's
functional equation
\begin{equation}\label{EQ1}f(xy)+\mu(y)f(\sigma(y)x)=2f(x)g(y),\,x,y\in S\end{equation}
and
\begin{equation}\label{EQ2}f(xy)+\mu(y)f(\sigma(y)x)=2f(y)g(x),\,x,y\in S\end{equation}
on a semigroup $S$ generated by its squares. We will show that the
functional equation (\ref{EQ1}) is related to the d'Alembert's
functional equation, and equation (\ref{EQ2}) is related to the
following variant of d'Alembert's functional equation
\begin{equation}\label{EQ21}g(xy)+\mu(y)g(\sigma(y)x)=2g(x)g(y)\end{equation}
solved on semigroups by Elqorachi and Redouani \cite[lemma
3.2]{Elqorachi and Redouani}. Note that $\sigma$ in the present
paper is a homomorphism, not an anti-homomorphism like the group
inversion found in other papers.
\\The general setting of (\ref{EQ1}) and (\ref{EQ2}) is for $S$ to be a semigroup
generated by its squares, because to formulate and solve the
functional equations (\ref{EQ1}) and (\ref{EQ2}) we need only an
associative composition in $S$ and the assumption that the semigroup
$S$ is generated by its squares, not an identity element and
inverses. Thus we study in the present paper (\ref{EQ1}) extending
works in which $S$ is a group or a monoid.
\par The organization of the paper is as follows. In the next
section we give notations and terminology. In section 3 we give some
preliminaries. In the fourth and fifth sections we prove our main
results.
\par The explicit formulas for the solutions are expressed in terms
of multiplicative and additive functions.
\section{Notations and Terminology}
\begin{defn} Let $f:S\to\mathbb{C}$ be a function on a
semigroup $S$. We say that\\
$f$ is additive if $f(xy)=f(x)+f(y)$ for all $x,y\in S$.\\
$f$ is multiplicative if $f(xy)=f(x)f(y)$ for all $x,y\in S$.\\
$f$ is central if $f(xy)=f(yx)$ for all $x,y\in S$.
\end{defn}
If $S$ is a semigroup, $\sigma:S\to S$ an involutive automorphism
and $\mu:S\to\mathbb{C}$ a multiplicative function such that
$\mu(x\sigma(x))=1$ for all $x\in S$, we define, for any function
$F:S\to\mathbb{C}$, the function
$$F^{*}(x)=\mu(x)F(\sigma(x)),\,x\in S.$$
Let $f:S\to\mathbb{C}$. We call $f^{e}:=\dfrac{f+f^{*}}{2}$ the even
part of $f$ and
$f^{o}:=\dfrac{f-f^{*}}{2}$ its odd part. The function $f$ is said to be even if $f=f^{*}$, and $f$ is said to be odd if $f=-f^{*}$.\\
\textbf{Blanket assumption:} Throughout this paper $S$ denotes a
semigroup (a set with an associative composition) generated by its
squares. The map $\sigma:S\to S$ denotes an involutive automorphism.
That $\sigma$ is involutive means that $\sigma(\sigma(x))=x$ for all
$x\in S$. We denote by $\mu:S\to\mathbb{C}$ a multiplicative
function such that $\mu(x\sigma(x))=1$ for all $x\in S$.
\section{Preliminaries}
The following Lemmas will be useful later.
\begin{lem}\label{lem610}Let $G$ be a semigroup, $\chi:G\to\mathbb{C}$ be a multiplicative
function, $\varphi:G\to G$ be a homomorphism and $F:G\to\mathbb{C}$
satisfy that $F(xy)+\chi(y)F(\varphi(y)x)=0$ for all $x,y\in G$.
Then $F(xyz)=0$ for all $x,y,z\in G$. $F=0$, if $G=\{xy\,|\,x,y\in
G\}$.
\end{lem}
\begin{proof} Let $x,y,z\in G$ be arbitrary. Applying the hypothesis
on $F$ repeatedly we get that
$F(xyz)=F(x(yz))=-\chi(yz)F(\varphi(yz)x)=-\chi(z)\chi(y)F(\varphi(y)(\varphi(z)x))=\chi(z)F((\varphi(z)x)y)=\chi(z)F(\varphi(z)(xy))=-F(xyz)$,
which implies that $2F(xyz)=0$, so that $F(xyz)=0$. If
$G=\{xy\,|\,x,y\in G\}$ then any element of $G$ can be written as
$xyz$, where $x,y,z\in G$, so $F=0$. This finishes the proof.
\end{proof}
\begin{lem}\label{lem611}Let $G$ be a semigroup such that $G=\{xy\,|\,x,y\in
G\}$. Let $f$ and $F$ be functions on $G$ such that $f(xy)=F(yx)$
for all $x,y\in G$. Then $f=F$.
\end{lem}
\begin{proof} For any $x,y,z\in G$ we have
$f(xyz)=f((xy)z)))=F(zxy)=F((zx)y)=f(y(zx))=f((yz)x)=F(x(yz))=F(xyz)$.
Applying the assumption on $G$ twice we see that any element of $G$
can be written as $xyz$, where $x,y,z\in G$. So $f=F$.
\end{proof}
\begin{lem}\label{lem612}\cite[lemma 3.2]{Elqorachi and Redouani}
Let $G$ be a semigroup, $\varphi:G\to G$
an involutive automorphism and $\chi:G\to \mathbb{C}$ be a
multiplicative function such that $\chi(x\varphi(x))=1$ for all
$x\in G$. The solutions $g$ of the functional equation
\begin{equation*}g(xy)+\mu(y)g(\sigma(y)x)=2g(x)g(y)\end{equation*}
on $G$ are of the form $g=\dfrac{m+\chi m\circ\varphi}{2}$, where
$m:G\to \mathbb{C}$ is a multiplicative function.
\end{lem}
\section{Solutions of Eq. (\ref{EQ1}) on a semigroup generated by its squares}
Recently, Elqorachi and Redouani \cite{Elqorachi and Redouani}
solved the generalized variant of Wilson's functional equation
(\ref{EQ1}) on groups with $\mu$ a group character. Fadli et al.
\cite{Fadli et al.} have some results on monoids with $\mu=1$. In
this section we extend the result in \cite{Elqorachi and Redouani}
to semigroups generated by their squares. We will need the following
Lemma relating the functional equation (\ref{EQ1}) to the
$\mu$-d'Alembert's functional equation.
\begin{lem}\label{lem64}Let $(f,g)$ be a solution of the functional
equation (\ref{EQ1}). Suppose that $f\neq0$ and $f$ is central. Then
$g$ satisfies the $\mu$-d'Alembert's functional equation
\begin{equation}\label{EQ4}g(xy)+\mu(y)g(x\sigma(y))=2g(x)g(y),\,x,y\in
S.\end{equation}
\end{lem}
\begin{proof} Since $f$ is central, $(f,g)$ satisfies the functional equation
\begin{equation}\label{EQ3}f(xy)+\mu(y)f(x\sigma(y))=2f(x)g(y),\,x,y\in S.\end{equation}
Moreover $f\neq0$, then there exists $x_{0}\in S$ such that
$f(x_{0})\neq0$. By replacing $x$ by $x_{0}$ and $y$ by $x$ in
(\ref{EQ3}) we obtain
$g(x)=\dfrac{1}{2f(x_{0})}[f(x_{0}x)+\mu(x)f(x_{0}\sigma(x))]$ for
all $x\in S$. It follows that
\begin{equation*}\begin{split}&g(xy)+\mu(y)g(x\sigma(y))=\dfrac{1}{2f(x_{0})}[f(x_{0}xy)+\mu(xy)f(x_{0}\sigma(x)\sigma(y))]\\
&+\dfrac{\mu(y)}{2f(x_{0})}[f(x_{0}x\sigma(y))+\mu(x\sigma(y))f(x_{0}\sigma(x)y))]\\&=\dfrac{1}{2f(x_{0})}[f(x_{0}xy)+\mu(y)f(x_{0}x\sigma(y))]
+\dfrac{\mu(x)}{2f(x_{0})}[f(x_{0}\sigma(x)y)+\mu(y)f(x_{0}\sigma(x)\sigma(y))]\\
&=\dfrac{1}{f(x_{0})}[f(x_{0}x)+\mu(x)f(x_{0}\sigma(x))]g(y)\\
&=\dfrac{2}{f(x_{0})}f(x_{0})g(x)g(y),\end{split}\end{equation*} for
all $x,y\in S$. Hence $g$ satisfies (\ref{EQ4}).
\end{proof}
\begin{thm}\label{thm65} The solutions $(f,g)$ of the functional equation
(\ref{EQ1}) are the following pairs:\\
(1) $f=0$ and $g$ arbitrary.\\
(2) $f=\lambda\chi+\delta\chi^{*}$ and $g=\dfrac{\chi+\chi^{*}}{2}$
where $\chi:S\to\mathbb{C}$ is a nonzero multiplicative function on
$S$ and $\lambda,\delta\in\mathbb{C}$ are constants such that
$(\lambda,\delta)\neq(0,0)$.\\
(3) \[ \left\{
\begin{array}{r c l}
f&=&\chi(c+A)\quad\text{and}\quad g=\chi\quad\text{on}\quad S\setminus I_{\chi}\\
f&=&0\quad\text{and}\quad g=0\quad\text{on}\quad I_{\chi},
\end{array}
\right.
\]
where $c\in\mathbb{C}$ is a constant, $\chi:S\to\mathbb{C}$ is a
nonzero multiplicative function and $A:S\setminus I_{\chi}\to
\mathbb{C}$ is a nonzero additive function such that $\chi=\chi^{*}$
and $A\circ\sigma=-A$.
\end{thm}
\begin{proof} Let $(f,g)$ be a solution of Eq. (\ref{EQ1}). If $f=0$ then $g$ is
arbitrary, which is solution (1).\\In what remains of the proof we
assume that $f\neq0$. Our computations are similar to those of
Stetk{\ae}r \cite{Stetkaer3}. Let $a,x,y\in S$ be arbitrary. \\For
the pair $(ax,y)$ the functional equation (\ref{EQ1}) implies that
$$f(axy)+\mu(y)f(\sigma(y)ax)=2f(ax)g(y).$$
By applying (\ref{EQ1}) to the pair $(\sigma(y)a,x)$ and multiplying
the identity obtained by $-\mu(y)$ we obtain
$$-\mu(y)f(\sigma(y)ax)-\mu(xy)f(\sigma(x)\sigma(y)a)=-2\mu(y)f(\sigma(y)a)g(x).$$
When we apply (\ref{EQ1}) to the pair $(a,xy)$ we get that
$$f(axy)+\mu(xy)f(\sigma(x)\sigma(y)a)=2f(a)g(xy).$$
By adding the three identities above we obtain
\begin{equation}\label{eq7-54}f(axy)=g(xy)f(a)+f(ax)g(y)+g(x)[f(ay)-2f(a)g(y)].\end{equation}
Using the notation
\begin{equation}\label{eq7-55}f_{a}(x)=f(ax)-f(a)g(x)\end{equation}
and seeing that $a,x,y$ are arbitrary, we deduce from (\ref{eq7-54})
that, for all $a\in S$, the pair $(f_{a},g)$ satisfies the sine
addition law
\begin{equation}\label{eq7-56}f_{a}(xy)=f_{a}(x)g(y)+f_{a}(y)g(x),\,x,y\in S.\end{equation}
There are two cases to consider.\\
Case 1. Suppose that $f_{a}=0$ for all $a\in S$. Then  we get from
(\ref{eq7-55}) that
\begin{equation}\label{eq7-57}f(xy)=f(x)g(y)\end{equation} for all $x,y\in
S$, which says that $f$ is a joint eigenvector for a representation
( the right regular representation) with $g(y)$, $y\in S$, as
corresponding eigenvalues. It follows that $g$ is multiplicative. We
write $\chi:=g$, and note that $\chi\neq0$ by Lemma \ref{lem610}.
Substituting (\ref{eq7-57}) in (\ref{EQ1}) we find that
\begin{equation}\label{eq7-571}f^{*}(y)\chi(x)=f(x)\chi(y)\quad\text{for all}\quad x,y\in S,\end{equation}
which implies that $f=\lambda\chi$ for some $\lambda\in \mathbb{C}$.
But $f\neq0$, so $\lambda\in\mathbb{C}\setminus\{0\}$. The result occurs in part (2) with $\delta=0$.\\
Case 2. Suppose that there exists $a\in S$ such that $f_{a}\neq0$.
Then, according to \cite[Theorem 4.1(b)]{Stetkaer2}, we deduce from
(\ref{eq7-56}) that there exist two multiplicative functions
$\chi_{1},\chi_{2}:S\to\mathbb{C}$ such that
\begin{equation}\label{eq7-58}g=\dfrac{\chi_{1}+\chi_{2}}{2}.\end{equation}
We split the discussion into the subcases of $\chi_{1}\neq\chi_{2}$ and $\chi_{1}=\chi_{2}$.\\
Subcase 2.1. Suppose that $\chi_{1}\neq\chi_{2}$. Then, according to
\cite[Theorem 4.1(c)]{Stetkaer2}, there exists a constant
$\alpha\in\mathbb{C}\setminus\{0\}$ such that
\begin{equation}\label{eq7-59}f_{a}=\dfrac{1}{2\alpha}(\chi_{1}-\chi_{2}).\end{equation}
Let $x,y,z\in S$ be arbitrary. For the pair $(x,yz)$ Eq. (\ref{EQ1})
reads
$$f(xyz)+\mu(yz)f(\sigma(y)\sigma(z)x)=2f(x)g(yz).$$
By applying Eq. (\ref{EQ1}) to the pair $(z,\sigma(x)\sigma(y))$ and
multiplying the identity obtained by $\mu(xy)$ we obtain
$$\mu(xy)f(z\sigma(x)\sigma(y))+f(xyz)=2\mu(xy)f(z)g(\sigma(x)\sigma(y)).$$
When we apply Eq. (\ref{EQ1}) to the pair $(\sigma(y),\sigma(z)x)$
and multiplying the identity obtained by $-\mu(yz)$ we get that
$$-\mu(yz)f(\sigma(y)\sigma(z)x)-\mu(xy)f(z\sigma(x)\sigma(y))=-2\mu(yz)f(\sigma(y))g(\sigma(z)x).$$
Now we add the identities above, which yields that
\begin{equation}\label{eq7-601}f(xyz)=\mu(xy)f(z)g(\sigma(x)\sigma(y))+f(x)g(yz)-\mu(yz)f(\sigma(y))g(\sigma(z)x).\end{equation}
For the triple $(x,y,z)$ the identity (\ref{eq7-54}) reads
$$f(xyz)=g(yz)f(x)+f(xy)g(z)+g(y)f(xz)-2f(x)g(y)g(z).$$ From the two
last identities we deduce that
\begin{equation}\label{eq7-60}\begin{split}&f(xy)g(z)+g(y)f(xz)-2f(x)g(y)g(z)=\mu(xy)f(z)g(\sigma(x)\sigma(y))\\&-\mu(yz)f(\sigma(y))g(\sigma(z)x),\end{split}\end{equation}
for all $x,y,z\in S$.\\We have two subcases to consider.\\
Subcase 2.1.1. Suppose that there exists $y_{0}\in S$ such that
$g(\sigma(a)y_{0})\neq0$. In view of (\ref{eq7-56}) we get,
according to \cite[Theorem 4.1(e)]{Stetkaer2}, that $g$ and $f_{a}$
are abelian. As $g(\sigma(a)y_{0})\neq0$ and $\mu(x)\neq0$ for all
$x\in S$, we derive from (\ref{eq7-60}), by putting
$(x,y)=(a,\sigma(y_{0}))$, that $f$ is abelian. Hence, according to
Lemma \ref{lem64}, $g$ is a solution of Eq. (\ref{EQ4}). So, seeing
that $g$ is central in view of (\ref{eq7-58}), we get that $g$
satisfies the functional equation
$$g(xy)+\mu(y)g(\sigma(y)x)=2g(x)g(y),\,x,y\in S.$$ Now, according to Lemma \ref{lem612}, there exists a multiplicative function
$\chi:S\to\mathbb{C}$ such that
\begin{equation}\label{eq7-61}g=\dfrac{\chi+\chi^{*}}{2}.\end{equation} From (\ref{eq7-61}) and (\ref{eq7-58})
we get that $\chi_{1}+\chi_{2}=\chi+\chi^{*}$. Since
$\chi_{1}\neq\chi_{2}$ we have $\chi^{*}\neq\chi$. Moreover,
according to \cite[Corollary 3.19]{Stetkaer2}, we get that
$\chi_{1}=\chi$ and $\chi_{2}=\chi^{*}$, or $\chi_{1}=\chi^{*}$ and
$\chi_{2}=\chi$. Hence, up to interchange $\chi_{1}$ and $\chi_{2}$,
the identity (\ref{eq7-58}) reads
\begin{equation}\label{eq7-62}f_{a}=\dfrac{1}{2\alpha}(\chi-\chi^{*}).\end{equation}
On the other hand, by replacing $x$ by $a$ and $y$ by
$\sigma(y_{0})$ in (\ref{eq7-60}) and using (\ref{eq7-55}) we get
that
\begin{equation*}\begin{split}&\mu(a\sigma(y_{0}))f(z)g(\sigma(a)y_{0})=[f(a\sigma(y_{0}))
-g(\sigma(y_{0}))f(a)]g(z)+g(\sigma(y_{0}))f_{a}(z)\\&+\mu(\sigma(y_{0}))f(y_{0})\mu(z)g(\sigma(z)a),\end{split}\end{equation*}
for all $z\in S$. Since $\mu(a\sigma(y_{0}))\neq0$,
$g(\sigma(a)y_{0})\neq0$ and $\mu(z\sigma(z))=1$ for all $z\in S$,
we derive from the identity above, by using (\ref{eq7-61}) and
(\ref{eq7-62}), that there exist two constants
$\lambda,\delta\in\mathbb{C}$ such that
$f=\lambda\chi+\delta\chi^{*}$. As $f\neq0$ we have
$(\lambda,\delta)\neq(0,0)$. The result obtained in this Subcase
occurs in part (2).\\
Subcase 2.1.2. Suppose that $g(\sigma(a)y)=0$ for all $y\in S$.
 Then, by putting $x=a$ in (\ref{eq7-601}) we get that
 $$f(ayz)=f(a)g(yz)-\mu(yz)f(\sigma(y))g(\sigma(z)a),$$ for all $y,z\in
 S$. According to (\ref{eq7-58}) and (\ref{eq7-59}) the functions $g$ and $f_{a}$
 are central, so is the function $x\mapsto f(ax),\,x\in S$ in view of
 (\ref{eq7-55}). So, the identity above implies that $\mu(yz)f(\sigma(y))g(a\sigma(z))=\mu(yz)f(\sigma(z))g(a\sigma(y))$
 for all $y,z\in S$. As $\mu(x)\neq0$ for all $x\in S$ and $\sigma$
 is involutive we deduce that $$f(y)g(az)=f(z)g(ay)$$ for all $y,z\in
 S$. Since $f\neq0$ we derive, by dividing the last identity by $f(y)$ with $y\in S$ chosen so that $f(y)\neq0$, that there
 exists a constant $\lambda$ such that
\begin{equation}\label{eq7-63}g(ax)=\lambda f(x),\end{equation}
 for all $x\in S$. We discuss the following subcases\\
Subcase 2.1.2.1. Suppose that $\lambda=0$. Then, (\ref{eq7-63})
implies that $g(ax)=0$ for all $x\in S$. As $g(\sigma(a)y)=0$ for
all $y\in S$ we get from (\ref{eq7-60}) that
\begin{equation}\label{eq7-64}f(ay)g(z)+g(y)f(az)=2g(y)g(z)f(a),\end{equation}
for all $y,z\in S$. Now, introducing $f_{a}$ by using
(\ref{eq7-55}), the identity (\ref{eq7-64}) reads
$f_{a}(y)g(z)+f_{a}(z)g(y)=0$ for all $y,z\in S$. According to
\cite[Exercise 1.1(b), p.6]{Stetkaer2} $f_{a}=0$ or $g=0$. In the
present Case 2 $f_{a}\neq0$, so $g=0$.
Then, according to Lemma \ref{lem610}, $f=0$ which contradicts the assumption on $f$. Hence, the present Subcase does not occur.\\
Subcase 2.1.2.2. Suppose that $\lambda\neq0$. Then, (\ref{eq7-63})
implies that
\begin{equation}\label{eq7-65}f(x)=\dfrac{1}{\lambda}g(ax),\end{equation}
for all $x\in S$, from which we derive, taking (\ref{eq7-58}) into
account, that $f$ is central. As $f\neq0$ we deduce by Lemma
\ref{lem64} and seeing that $g$ is central, that $g$ satisfies the
functional equation
$$g(xy)+\mu(y)g(\sigma(y)x)=2g(x)g(y),\,x,y\in S.$$
So, according to Lemma \ref{lem612}, there exists a multiplicative
function $\chi:S\to\mathbb{C}$ such that
$$g=\dfrac{\chi+\chi^{*}}{2}$$ and, as in Subcase 2.1.1.,
$\chi^{*}\neq\chi$. Hence, the supposition $g(\sigma(a)x)=0$ for all
$x\in S$ implies that $\chi(a)\chi+\chi^{*}(a)\chi^{*}=0$. So,
according to \cite[Theorem 3.18(a)]{Stetkaer2}, we get that
$\chi(a)=\chi^{*}(a)=0$. By substituting this in the expression
above of $g$, we obtain $g(ax)=0$ for all $x\in S$. So, taking
(\ref{eq7-65}) into account, we get that $f=0$, which contradicts
the assumption on $f$. Hence, the present Subcase does
not occur.\\
Subcase 2.2. Suppose that $\chi_{1}=\chi_{2}$. Then letting
$\chi:=\chi_{1}$ we have $g=\chi$. Note for use bellow that
$\chi\neq0$. Since $S$ is a semigroup generated by its squares and
$f_{a}\neq0$ we get, according to \cite[Lemma 3.4(ii)]{Ebanks and
Stetkaer2}, that there exists a nonzero additive function
$A:S\setminus I_{\chi}\to \mathbb{C}$ such that
\begin{equation}\label{eq7-66}
f_{a}=\chi A\quad\text{on}\quad S\setminus
I_{\chi}\quad\text{and}\quad f_{a}=0\quad\text{on}\quad I_{\chi}.
\end{equation}
Proceeding exactly as we have done to prove (\ref{eq7-59}) in
Subcase 2.1 we get after simplifications, by replacing $x$ by $a$,
that
\begin{equation}\label{eq7-67}f_{a}(yz)=\chi^{*}(a)\chi^{*}(y)f(z)-\chi(a)f^{*}(y)\chi^{*}(z),\end{equation}
for all $y,z\in S$. According to \cite[Theorem 4.4(e)]{Stetkaer2}
the function $f_{a}$ is abelian. Moreover, since $f_{a}\neq0$ and
$S$ is generated by its squares, the identity (\ref{eq7-67}) implies
that $\chi^{*}(a)\neq0$ or $\chi(a)\neq0$. Since $\sigma:S\to S$ is
involutive and $\chi\neq0$ there exists $y_{0}\in S$ such that
$\chi(\sigma(y_{0}))\neq0$. Now, when we put $y=y_{0}$ in
(\ref{eq7-67}) if $\chi^{*}(a)\neq0$, and $z=y_{0}$ if
$\chi(a)\neq0$, and using that $\mu(x)\neq0$ for all $x\in S$, we
derive that $f$ is central. As $f\neq0$ and $g$ is central we get,
by applying Lemma \ref{lem64}, that $g$ satisfies the functional
equation
$$g(xy)+\mu(y)g(\sigma(y)x)=2g(x)g(y),\,x,y\in S.$$ So that,
according to Lemma \ref{lem612}, there exists a multiplicative
function $m:S\to\mathbb{C}$ such that $g=\dfrac{m+m^{*}}{2}$, which
implies that $2\chi=m+m^{*}$. Hence $\chi=m=m^{*}=\chi^{*}$. So, by
putting $y=a$, the identity (\ref{eq7-67}) reduces to
\begin{equation}\label{eq7-68}f_{a}(az)=[\chi(a)f(z)-\mu(a)f(\sigma(a))\chi(z)]\chi(a),\end{equation}
for all $z\in S$. We have seen above that either $\chi^{*}(a)=0$ or
$\chi(a)=0$. As $\chi=\chi^{*}$ we get that $\chi(a)\neq0$. By a
small computation we deduce, by using (\ref{eq7-66}) and
(\ref{eq7-68}) and writing $A$ instead of $(\chi(a))^{-1}A$, that
\begin{equation}\label{eq7-69}
f=\chi(c+ A)\quad\text{on}\quad S\setminus
I_{\chi}\quad\text{and}\quad f=0\quad\text{on}\quad I_{\chi},
\end{equation}
where $c\in\mathbb{C}$ is a constant. Since $\chi^{*}=\chi$ we get
easily  that $\sigma(S\setminus I_{\chi})=S\setminus I_{\chi}$.
Hence, the functional equation (\ref{EQ1}) reads
$\chi(xy)(c+A(xy))=2\chi(xy)(c+A(x))$ for all $x\in S\setminus
I_{\chi}$, which implies that $A(x)=A(y)-c$ for all $x\in S\setminus
I_{\chi}$. Then the additive function $A$ is bounded on $S\setminus I_{\chi}$, which contradicts that $A\neq0$.\\
Now, let $x,y\in S\setminus I_{\chi}$ be arbitrary. The functional
equation (\ref{EQ1}) implies that
$$\chi(xy)(c+A(xy))+\mu(y)\chi(\sigma(y)x)(c+A(\sigma(y)x)=2\chi(y)\chi(x)(c+A(x)).$$
By using that $\chi^{*}=\chi$, $\chi$ is multiplicative, $A$ is
additive and $\chi(xy)\neq0$ the identity above reduces, by a small
computation, to $(A\circ\sigma+A)(y)=0$. So, $y$ being arbitrary, we
deduce that $A\circ\sigma=-A$. The result
obtained in this Subcase occurs in part (3).\\
Conversely, if $f$ and $g$ are of the forms (1)-(3) in Theorem
\ref{thm65} we check by elementary computations that $f$ and $g$
satisfy the functional equation (\ref{EQ1}). This completes the
proof of Theorem \ref{thm65}.
\end{proof}
\section{Solutions of Eq. (\ref{EQ2}) on a semigroup generated by its squares}
In Theorem \ref{thm66} we solve the functional equation (\ref{EQ2})
on semigroups generated by their squares.
\begin{thm}\label{thm66}The solutions $f,g,:S\to\mathbb{C}$ of the functional equation
(\ref{EQ2}) are the follows pairs:\\
(1) $f=0$ and $g$ is arbitrary.\\
(2) $f=\alpha\dfrac{\chi+\chi^{*}}{2}$ and
$g=\dfrac{\chi+\chi^{*}}{2}$, where
$\alpha\in\mathbb{C}\setminus\{0\}$ is a constant and $\chi:S\to
\mathbb{C}$ is a multiplicative function.
\end{thm}
\begin{proof}An elementary computation shows that if $f$ and $g$ are
the forms (1)-(2) then they satisfy the functional equation
(\ref{EQ2}), so left is that any solution $(f,g)$ of (\ref{EQ2})
fits into (1)-(2).\\Let $f,g:S\to\mathbb{C}$ be a solution of Eq.
(\ref{EQ2}). The part (1) is trivial. So, in what follows we assume
that $f\neq0$. As $S$ is generated by its squares, we get by using
Lemma \ref{lem610} that $g\neq0$.\\Let $x,y\in S$ be arbitrary. By
applying (\ref{EQ2}) to the pair $(\sigma(y),x)$ and multiplying the
identity obtained by $\mu(y)$ we get that
$$\mu(y)f(\sigma(y)x)+f^{*}(xy)=2f(x)g^{*}(y).$$
By subtracting this from (\ref{EQ2}) we obtain
\begin{equation}\label{eq7-70}f^{o}(xy)=g(x)f(y)-f(x)g^{*}(y).\end{equation}
On the other hand, seeing that $f=f^{e}+f^{o}$ and taking
(\ref{eq7-70}) into account we get, from the functional equation
(\ref{EQ2}), that
\begin{equation*}\begin{split}&f^{e}(xy)+\mu(y)f^{e}(\sigma(y)x)=2f(y)g(x)-f^{o}(xy)-\mu(y)f^{o}(\sigma(y)x)\\
&=2g(x)f(y)-f(y)g(x)+f(x)g^{*}(y)-\mu(y)[g(\sigma(y))f(x)-f(\sigma(y))g^{*}(x)]\\
&=g(x)f(y)-f(x)g^{*}(y)+g^{*}(y)f(x)-f^{*}(y)g^{*}(x).\end{split}\end{equation*}
So that
\begin{equation*}f^{e}(xy)+\mu(y)f^{e}(\sigma(y)x)=g(x)f(y)+g^{*}(x)f^{*}(y).\end{equation*}
Since $f^{*}=f^{e}-f^{o}$ the identity above becomes
\begin{equation*}f^{e}(xy)+\mu(y)f^{e}(\sigma(y)x)=[g^{e}(x)+g^{o}(x)][f^{e}(y)+f^{o}(y)]+[g^{e}(x)-g^{o}(x)][f^{e}(y)-f^{o}(y)],\end{equation*}
which implies, by a small computation, that
\begin{equation}\label{eq7-71}f^{e}(xy)+\mu(y)f^{e}(\sigma(y)x)=2g^{e}(x)f^{e}(y)+2g^{o}(x)f^{o}(y).\end{equation}
For the pair $(\sigma(x),y)$ the functional equation (\ref{eq7-71})
reads
$$f^{e}(\sigma(x)y)+\mu(y)f^{e}\circ\sigma(yx)=2g^{e}(\sigma(x))f^{e}(y)+2g^{o}(\sigma(x))f^{o}(y).$$
Now, when we multiply this by $\mu(x)$ we obtain
\begin{equation}\label{eq7-72}\mu(x)f^{e}(\sigma(x)y)+f^{e}(yx)=2g^{e}(x)f^{e}(y)-2g^{o}(x)f^{o}(y).\end{equation}
So, $x$ and $y$ being arbitrary, we get from (\ref{eq7-71}) and
(\ref{eq7-72}) that
\begin{equation*}g^{e}(x)f^{e}(y)-g^{e}(y)f^{e}(x)=g^{o}(x)f^{o}(y)+g^{o}(y)f^{o}(x)\end{equation*}
for all $x,y\in S$.\\In the identity above the left hand side is an
even function of $x$ while the right is an odd function of $x$, then
\begin{equation}\label{eq7-73}g^{e}(x)f^{e}(y)-g^{e}(y)f^{e}(x)=0\end{equation}
and
\begin{equation}\label{eq7-74}g^{o}(x)f^{o}(y)+g^{o}(y)f^{o}(x)=0,\end{equation}
for all $x,y\in s$. According to \cite[Exercise 1.1(b),
p.6]{Stetkaer2} we deduce from (\ref{eq7-74}) that we have two cases
to consider\\
\underline{Case 1}: $f^{o}=0$. As $g\neq0$ we get from
(\ref{eq7-70}) that
\begin{equation}\label{eq7-75}g=\lambda f,\end{equation} for some constant
$\lambda\in\mathbb{C}\setminus\{0\}$. Hence, multiplying Eq.
(\ref{EQ2}) by $\lambda$, the function $g$ satisfies the following
variant of d'Alembert's functional equation
$$g(xy)+\mu(y)g(\sigma(y)x)=2g(x)g(y),$$ for all $x,y\in S$. By
applying Lemma \ref{lem612} and taking (\ref{eq7-75}) into account
we obtain part (2).\\
\underline{Case 2}: $g^{o}=0$. Then $g^{*}=g$ and (\ref{eq7-70})
reads $f^{o}(xy)=g(x)f(y)-f(x)g(y)=-f^{o}(yx)$ for all $x,y\in S$.
Sine $S$ is generated by its squares we deduce by applying Lemma
\ref{lem611} that $f^{o}=0$. So we go back to Case 1. This finishes
the proof.
\end{proof}

\end{document}